\title{The \LaTeX\  template file for an article in the
       $\Pi$ME Journal}
\author{Primus Scriber\thanks{
                  College of the Enlightenment}
        \and
        Theeco Author\thanks{
                      The Virtual University
                  }\thanks{This work was
                  supported by NSB grant number G983578765401.}
        }
\documentclass{article}
\usepackage{graphicx}
\usepackage{amsthm}
\usepackage{tikz}
\usetikzlibrary{er,positioning,bayesnet}

\usepackage{float}
\newtheorem{theorem}{Theorem}
\newtheorem{lemma}{Lemma}

\newtheorem{defn}{Definition}
\newtheorem{corollary}{Corollary}
\newtheorem{algo}{Algorithm}

\newenvironment{keywords}{}{}
\title{A Lower Bound on the Failed Zero Forcing Number of a Graph}
\author{Eric Ufferman\thanks{
                  Department of Mathematics, Virginia Tech University \quad ericu1@vt.edu}
        \and
        Nicolas Swanson\thanks{
                      Virginia Tech University \quad nicswanson@vt.edu
                  }
        }

\begin{document}
%
%
%

%
%
 
%
\maketitle              
\begin{abstract}

Given a graph $G=(V,E)$ and a set of vertices marked as filled, we consider a color-change rule known as zero forcing. A set $S$ is a zero forcing set if filling $S$ and applying all possible instances of the color change rule causes all vertices in  $V$ to be filled. A failed zero forcing set is a set of vertices that is not a zero forcing set. Given a graph $G$, the failed zero forcing number $F(G)$ is the maximum size of a failed zero forcing set. In \cite{fetcie}, the authors asked whether given any $k$
 there is a an $\ell$ such that all graphs with at least $\ell$ vertices must satisfy $F(G)\geq k$. We answer this question affirmatively by proving that for a graph $G$ with $n$ vertices, $F(G)\geq \lfloor\frac{n-1}{2}\rfloor$.

\end{abstract}

\section{Introduction}

Given a graph $G=(V,E)$ and a subset $S\subseteq V$ of filled vertices, we consider the following the color change rule. If any filled vertex $v$ is adjacent to exactly one unfilled vertex $w$, then $v$ ``forces'' $w$, and  we mark $w$ as filled. The set of filled vertices after all possible instances of the rule are applied is called the derived set of $S$. If the derived set of $S$ is all of $V$, we say that $S$ is a zero forcing set for $G$. Otherwise it is a failed zero forcing set. The zero forcing number of a graph $G$, denoted $Z(G)$ is the minimum size of any zero forcing set of $G$. The failed zero forcing number $G$, denoted $F(G)$ is the maximum size of any failed zero forcing set of $G$. Zero forcing numbers and their relation to minimum rank problems have been studied extensively, for example in \cite{barioli}. Failed zero forcing numbers were introduced and studied in \cite{fetcie}.

Any path $P_n$ satisfies $Z(P_n)=1$. This can be seen by taking the initial set of filled vertices to consist of exactly one end vertex of the path. The end vertex forces its neighbor, which in turn forces its unfilled neighbor, and so on until the entire path is filled. Hence there are graphs with arbitrarily large vertex set whose zero forcing number is 1. 

In \cite{fetcie}, the authors asked whether a similar property holds for the failed zero forcing number. Intuition tells us that for any fixed integer $k$ there should not be arbitrarily large graphs satisfying $F(G)=k$. The larger the graph, the more room there is to distribute a fixed number of filled vertices sparsely in such a way that no filled vertex can force. We confirm this intuition by proving that for any graph $G$ on $n$ vertices $F(G)\geq\lfloor \frac{n-1}{2}\rfloor$. This is the best bound possible, as it was shown in \cite{fetcie} that any path $P_n$  satisfies
$F(P_n)=\lfloor\frac{n-1}{2}\rfloor$. Strikingly, for any given number of vertices, the minimum possible zero forcing number and maximum possible failed zero forcing number are both obtained by paths. 

The bound makes it possible in principle to classify all graphs with a given failed zero forcing number $k$ by exhaustively checking all graphs such that $n\leq 2k + 2$. This classification was done in \cite{fetcie} for $k=0,1$, and in \cite{gomez} for $k=2$, the latter using a special case of the bound proven here.

In what follows we prove our main result by showing a somewhat stronger result for graphs with minimum vertex degree of at least three. We then show the desired bound inductively for graphs with minimum degree one and two.

\section{Graphs with minimum degree at least three}

The main result of this section is the following lower bound on the failed zero forcing number of a graph of minimum degree three or more:

\begin{theorem}\label{delta3}
Let $G=(V,E)$ be a connected graph on $n$ vertices satisfying $\delta(G)\geq 3$. Then $F(G) \geq \lfloor \frac{n}{2}\rfloor$. If $G$ contains a circuit of even length (and $n$ is odd), we can improve the bound  to $F(G) \geq \lceil\frac{n}{2}\rceil$.
\end{theorem}

\tikzset{
unfilled/.style = {circle, draw = black, minimum size = 0.5cm},
filled/.style = {unfilled, fill = blue!40, minimum size = 0.5cm},
collection/.style = {ellipse, draw = black, thick}
}

We develop  machinery needed to prove this result, starting with some terminology. If a set of vertices $S\neq V$ is its own derived set, then we say that $S$ is stalled. Note that any maximal failed zero forcing set is stalled. We will distinguish between vertices that cannot force because all of their neighbors are filled and vertices that cannot force because they have multiple unfilled neighbors.
\begin{defn}
Let $G=(V,E)$, $S\subseteq V$ be a set of filled vertices, and  $v\in S$. We say $v$ is \emph{spent} if all neighbors of $v$ are in $S$. Otherwise, $v$ is \emph{unspent}.
\end{defn}

\noindent\textbf{Observation:} If we have a maximal failed zero forcing set $S$ in $G$ in which all vertices are unspent, then all filled vertices are adjacent to at least two unfilled vertices. In this case, we can add any number of edges to $G$, and $S$ will still be a failed zero forcing set. This is a special case of lemma 3 in \cite{gomez}.

We next have a result about  bipartite graphs.

\begin{lemma}\label{bipartite}
If $G=(V,E)$ is a bipartite graph with bipartition $V=L\sqcup R$, and $\delta(G)\geq 2$, then $F(G)\geq\max\{|L|,|R|\}$. In particular, $F(G)\geq\lceil\frac{n}{2}\rceil$.
\end{lemma}

\begin{proof}
If we fill all the vertices in, say, $L$, then every filled vertex is adjacent to at least two (unfilled) vertices in $R$, and therefore cannot force. 
\end{proof}

Note that all vertices in the failed zero forcing set obtained in lemma (\ref{bipartite}) are unspent. Combining the lemma with the observation we have:

\begin{corollary}
If $G$ has a bipartite spanning subgraph $H$ such that $\delta(H)\geq 2$, then $F(G)\geq\lceil\frac{n}{2}\rceil$
\end{corollary}

A vertex cut in a connected graph $G$ is a set of vertices whose
removal causes $G$ to be disconnected or trivial. The connectivity of a connected graph $G$, denoted $\kappa(G)$, is the size of the smallest vertex cut.

We need the following:

\begin{lemma}\label{cut_vert}
Let $G=(V,E)$ be a connected graph on $n$ vertices satisfying $\delta(G)\geq 3$ and $\kappa(G)=1$. Then $F(G)\geq \lfloor\frac{n+1}{2}\rfloor$

\end{lemma}
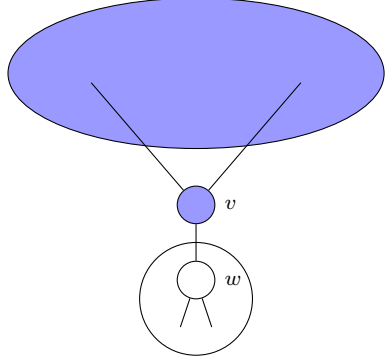
\begin{figure}[H]
    \caption{The graph $G$ with the vertices in $S \cup \{v\}$ colored.}
    \label{fig:cut_vert}
    \begin{tikzpicture}
        \draw[fill = blue!40] (0, 0) ellipse(2.5cm and 1cm);
        \draw (0, -3) ellipse(0.75cm and 0.75cm);

        \node[filled, label=right:{$v$}] (v) at (0, -1.75){};
        \node (d2) at (-1.5, 0){};
        \node (d3) at (1.5, 0){};
        \node (d1) at (0, -3){};
        \node[unfilled, label=right:{$w$}] (w) at (0, -2.75){};
        \node (w1) at (-.25, -3.5){};
        \node (w2) at (.25, -3.5){};
    
        \draw (v) -- (w);
        \draw (v) -- (d2);
        \draw (v) -- (d3);
        \draw (w) -- (w1);
        \draw (w) -- (w2);
    \end{tikzpicture}
    \centering
\end{figure}
 
\begin{proof}
Since $\kappa(G) = 1$ there is a vertex cut consisting of a single  vertex $v$. Let $G'$ be the graph induced by $V-\{v\}$. Since $G'$ is disconnected, we can label $k$ connected components of $G'$, $D_1, D_2,...,D_k$ with corresponding vertex sets $Y_1, Y_2,...,Y_k$, in order of increasing size. 
Let $S = \bigcup_{i = 2}^k Y_i$. Note that $|S| + 1 \ge \lfloor \frac{n+1}{2} \rfloor$.
  
The only vertex in $V - D_1$  adjacent to a vertex in $D_1$ is $v$. No vertex in $S$ can force a vertex in $D_1$, so we need only consider the color changing rule applied to $v$. If $v$ is adjacent to at least two vertices in $D_1$, then $S \cup \{v\}$, is stalled. If $v$ is adjacent to exactly one vertex in $D_1$, call it $w$ (see Figure \ref{fig:cut_vert}). Since $\delta(G) \ge 3$, $w$ has at least two neighbors in $D_1$. The derived set of $S \cup \{v\}$ would then include $w$, but not $w$'s neighbors in $D_1$ since there would be at least two of them. In either case, the derived set of $S \cup v$ excludes vertices in $D_1$, making it a failed zero forcing set. Thus, $F(G) \ge |S \cup \{v\}| \ge |S| + 1 \ge \lfloor \frac{n+1}{2} \rfloor$. 
 
 \end{proof}
 
 Now we need to handle the case where $G$ has no cut vertices. We give an algorithm for finding a failed zero forcing set with at least $\lfloor \frac{n-1}{2}\rfloor$ vertices in any graph $G$ satisfying $\delta(G)\geq 3$ and $\kappa(G)\geq 2$. The algorithm for produces a stalled zero forcing set of the required size, all of whose vertices are unspent.

\begin{algo}\label{algo}     Our strategy will be to partition $V$ into  three subsets $L$, $R$ and $O$ such that any vertex in $L$ is adjacent to at least two vertices in $R\cup O$, and any vertex in $R$ is adjacent to at least two vertices in $L\cup O$. We may then obtain a stalled set by filling either all the vertices in $L$, or all of the vertices in $R$. We will have that $|O|= 0$ or $|O|=1$, so that we are guaranteed to be able to fill at least $\lfloor \frac{n-1}{2}\rfloor$  vertices by choosing the larger of $L$ and $R$. In what follows, we refer to $L$ and $R$ as the ``sides'' of the partition. At any stage of the algorithm we let $A=L\cup R\cup O$ be the set of vertices that have already been assigned to one set in the partition, and we use $U$ the set of all vertices that remain unassigned.

Since $\delta(G)\geq 3$, $G$ has at least one cycle. If $G$ has a cycle $C
$ of even length, then assign alternating vertices of $C$ to $L$ and $R$ and let $O=\emptyset$. This guarantees that all assigned vertices are adjacent to two vertices on the other side of the partition.
If $G$ has no even cycles, then choose any odd cycle $C$ and let $O=\{v_0\}$ for an arbitrary chosen $v_0\in C$. Starting from either neighbor of $v_0$ in $C$ assign alternating vertices to $L$ and $R$. The neighbors of $v_0$ are each adjacent to a vertex in $O$ and a vertex on the other side of the partition, and all other vertices in $C$ are adjacent to two vertices on the other side of the partition.

Now, while any vertices remained unassigned we find a new vertex or vertices to assign by checking the following conditions in order of priority:

\begin{enumerate}
\item If there is any vertex $v$  that is adjacent to at least two vertices in $L\cup O$, we assign  $v$ to $R$.
\item If there is any   $v$ that is adjacent to at least two vertices in $R\cup O$, we assign   $v$ to $L$.
\item If there is any $v$ that is adjacent to exactly two vertices $w_1,w_2$ in $A$, where $w_1\in L$ and $w_2\in R$, we proceed as follows. Note that there must be a path $u_0,u_1,\ldots,v$ from some other vertex $u_0\in A$ to $v$, or else $v$ would be a cut vertex. We may assume the interior vertices of the path are in $U$. We assign $u_1$ to the opposite side of the partition as $u_0$ (or either side if $u_0\in O$), and proceed to alternate assignments until we reach (and assign) $v$. Interior assigned vertices are assigned two neighbors on the opposite side of the partition. The vertex $v $ itself has one neighbor that was just assigned to the opposite side of the partition. Together with one of its previously assigned neighbors, it now has two neighbors with the opposite assignment as itself.

\begin{figure}[H]
  \caption{Case 3 of Algorithm. A red label indicates the label was newly assigned in this stage. }\label{case3}
  \centering
    \begin{tikzpicture}
\tikzstyle{vertex} = [circle, draw=black]
\tikzstyle{filled_vertex} = [circle, draw=black, fill = blue!40]

\draw(0,0) ellipse(3.5cm and 1.75cm);

\node (v0) at (0,1){$A$};
\node[vertex,label=above:{$w_1$}] (v1) at (-1.5,-1){\tiny L};
\node[vertex,label=above:{$w_2$}] (v2) at (-.5,-1){\tiny R};
\node[vertex,text=red,label=below:{$v$}] (v3) at (-1,-2.25){\tiny R};
\node[vertex,text=red] (v4) at (0,-2.25){\tiny L};
\node[vertex,text=red] (v5) at (1,-2.25){\tiny R};
\node[vertex,text=red,label=below:{$u_1$}] (v6) at (2,-2.25){\tiny L};
\node[vertex,label=above:{$u_0$}] (v7) at (2,-.75){\tiny R};

\draw (v1) -- (v3);
\draw (v2) -- (v3);
\draw (v3) -- (v4);
\draw (v4) -- (v5);
\draw (v6) -- (v5);
\draw (v6) -- (v7);

\end{tikzpicture}
\end{figure}
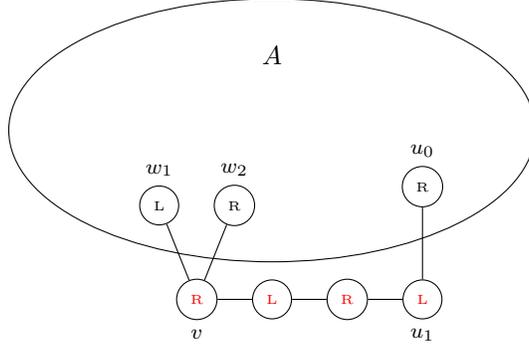

\item If none of the previous three cases apply at any stage, we have that all vertices in $U$ that have adjacencies in $A$ have exactly one neighbor in $A$. Let $H$ be the subgraph induced by $U$. Note that  $\delta(H)\geq 2$, so $H$ contains a cycle. If $H$ contains a cycle of even length, then we may simply assign vertices in the cycle to alternating sides of the partition, as in the initial step of the algorithm. So we may assume that all cycles have odd length. Choose any cycle $C$ in $H$. Let $P$ be a path from any vertex $c_0\in C$ to a vertex $v_U\in U$,  such that $v_u$ is adjacent to a vertex $v_A\in A$. We may choose $P$ so that $P\cap C=\{c_0\}$ and $P\cap A =\{v_A\}$. (It is possible that $c_0=v_U$.) Note that there must be a path $Q$ from some vertex $c_i\in C$,  $c_i\neq c_0$, to a vertex $w_U\in U$ that is adjacent to a vertex $w_A\in A$, or else $c_0$ would be a cut vertex of $G$.  We may also assume $Q\cap C=\{c_i\}$ and $Q\cap A=\{w_A\}$. It is possible that $v_A=w_A$. However, $P\cap Q=\emptyset$, for otherwise we would be able to produce two cycles of opposite parity containing $v_U$ using $P$, $Q$ and the two paths from $c_0$ to $c_i$ in $C$, and hence $U$ would contain an even cycle. We have two paths from $v_U$ to $w_U$ in $H$ via $c_0$ and $c_i$, taking opposite paths around $C$. Since $C$ is of odd length, the paths' lengths have opposite parity. If $v_A$ and $w_A$ are on opposite side of the partition, then give $v_U$ the opposite label of $v_A$, and proceed along the even-length path alternating labels.  If $v_A$ and $w_A$ are on opposite side of the partition, then give $v_U$ the opposite label of $v_A$, and proceed along the odd-length path alternating labels. 
\end{enumerate}

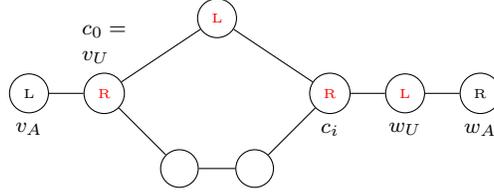
\begin{figure}[H]
  \caption{Case 4 of Algorithm. In the event that $v_A$ and $w_A$ have opposite labels, we alternate labels along an odd-length path  $v_A$ and $w_A$  }
  \centering
    \begin{tikzpicture}
\tikzstyle{vertex} = [circle, draw=black, minimum size = .5cm]
\tikzstyle{filled_vertex} = [circle, draw=black, fill = blue!40]

\node[vertex,label=below:{$v_A$}] (v1) at (0,0){\tiny L};
\node[vertex,text=red,label={[align=left]$c_0=$\\$v_U$}] (v2) at (1,0){\tiny R};
\node[vertex,text=red] (v3) at (2.5,1){\tiny L};
\node[vertex,text=red] (v4) at (2,-1){};
x\node[vertex,text=red] (v5) at (3,-1){};
\node[vertex,text=red,label=below:{$c_i$}] (v6) at (4,0){\tiny R};
\node[vertex,text=red,label=below:{$w_U$}] (v7) at (5,0){\tiny L};
\node[vertex,label=below:{$w_A$}] (v8) at (6,0){\tiny R};

\draw (v1)--(v2);
\draw (v2)--(v3);
\draw (v2)--(v4);
\draw (v4)--(v5);
\draw (v6)--(v5);
\draw (v3)--(v6);
\draw (v7)--(v6);
\draw (v7)--(v8);

\end{tikzpicture}
\end{figure}

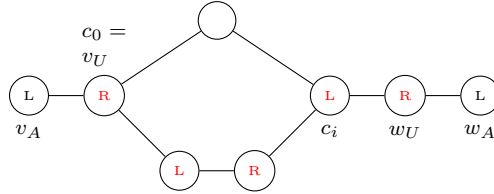
\begin{figure}[H]
  \caption{Case 4 of Algorithm. In the event that $v_A$ and $w_A$ have the same label, we alternate labels along an even-length path  $v_A$ and $w_A$  }
  \centering
    \begin{tikzpicture}
\tikzstyle{vertex} = [circle, draw=black, minimum size = .5cm]
\tikzstyle{filled_vertex} = [circle, draw=black, fill = blue!40]

\node[vertex,label=below:{$v_A$}] (v1) at (0,0){\tiny L};
\node[vertex,text=red,label={[align=left]$c_0=$\\$v_U$}] (v2) at (1,0){\tiny R};
\node[vertex,text=red] (v3) at (2.5,1){};
\node[vertex,text=red] (v4) at (2,-1){\tiny L};
\node[vertex,text=red] (v5) at (3,-1){\tiny R};
\node[vertex,text=red,label=below:{$c_i$}] (v6) at (4,0){\tiny L};
\node[vertex,text=red,label=below:{$w_U$}] (v7) at (5,0){\tiny R};
\node[vertex,label=below:{$w_A$}] (v8) at (6,0){\tiny L};

\draw (v1)--(v2);
\draw (v2)--(v3);
\draw (v2)--(v4);
\draw (v4)--(v5);
\draw (v6)--(v5);
\draw (v3)--(v6);
\draw (v7)--(v6);
\draw (v7)--(v8);

\end{tikzpicture}
\end{figure}

 Any time we make new assignments, we repeat the process, checking which case applies in order. The cases are exhaustive and in any of the cases we can assign labels to more vertices while maintaining the property that any labeled vertex has as least two neighbors each of which either has the opposite label or is labeled $O$. Hence the process will terminate with all vertices assigned a label and the final labeling satisfying the required property.
 \end{algo}

In the case where $G$ contains an even cycle, note that $O=\emptyset$, and all vertices  are labeled either $L$ or $R$ upon termination. Therefore $F(G)\geq \lceil\frac{n}{2}\rceil$ by Lemma \ref{bipartite}.

 Lemma \ref{cut_vert} and Algorithm \ref{algo} exhaust all possibilities for connected graphs with minimum degree at least three. Hence, we have proven Theorem \ref{delta3}.

\section{The General Case}

Next we handle the case where $G$ has vertices of degree less than three with an inductive argument.

\begin{theorem}
Let $G=(V,E)$ be a graph on $n$ vertices. $F(G)\geq \lfloor\frac{n-1}{2}\rfloor$
\end{theorem}

\begin{proof}

If $G$ is disconnected, we may find a failed zero set $S$ satisfying  $|S|\geq \lceil\frac{n}{2}\rceil$ by filling all vertices in all but the smallest connected component of $G$.  If $G$ is connected and $\delta(G)\geq 3$, we know from Theorem \ref{delta3} that the desired inequality holds.  So we only need to consider the cases where $\delta(G)=1$ and $\delta(G)=2$.

We proceed by induction on $n$. If $n=1$, $F(G)=0=\lfloor\frac{n-1}{2}\rfloor$ and the result holds. Now assume the inequality holds for all graphs with fewer than $n$ vertices.

Suppose $\delta(G) = 1$ and let $v \in V$ with $\deg(v) = 1$. Let $v$ be adjacent to $w$ and  let $G'$ be the subgraph induced by $V' = V-\{v,w\}$ .   By the inductive hypothesis,  $F(G') \ge \lfloor \frac{(n - 2) - 1}{2} \rfloor$. Let $S'\subset V'$ be a stalled set that achieves this bound.  

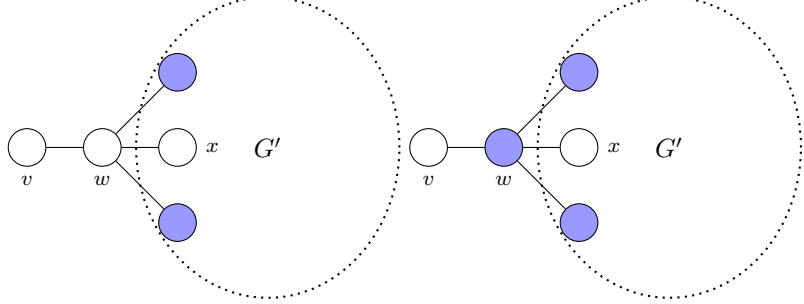
\begin{figure}[H]
    \caption{The graph $G$ with $\delta(G) = 1$. Case where $w$ is adjacent to $x \in V' - S'$.}
    \centering
    \begin{tikzpicture}
        \node[unfilled, label=below:{$v$}] at (0,0) (v) {};
        \node[unfilled, label=below:{$w$}] at (1,0) (w) {};
        \node[filled] at (2, -1) (w_1) {};
        \node[unfilled, label=right:{$x$}] at (2, 0) (w_2) {};
        \node[filled] at (2, 1) (w_3) {};

        \draw (v) -- (w);
        \draw (w) -- (w_1);
        \draw (w) -- (w_2);
        \draw (w) -- (w_3);

        \node[collection, minimum height = 4cm, minimum width = 3.5cm, dotted] at (3.2, 0) (W) {$G'$};
    \end{tikzpicture}
    \begin{tikzpicture}
        \node[unfilled, label=below:{$v$}] at (0,0) (v) {};
        \node[filled, label=below:{$w$}] at (1,0) (w) {};
        \node[filled] at (2, -1) (w_1) {};
        \node[unfilled, label=right:{$x$}] at (2, 0) (w_2) {};
        \node[filled] at (2, 1) (w_3) {};

        \draw (v) -- (w);
        \draw (w) -- (w_1);
        \draw (w) -- (w_2);
        \draw (w) -- (w_3);

        \node[collection, minimum height = 4cm, minimum width = 3.5cm, dotted] at (3.2, 0) (W) {$G'$};
    \end{tikzpicture}
\end{figure}

Suppose $w$ is adjacent to an unfilled vertex $x\in V'$. We claim $S = S' \cup \{w\}$ is stalled in $G$. Consider $y \in S'$. We see $y$ is either only adjacent to vertices in $S' \cup \{w \}$, or it is adjacent to at least 2 unfilled vertices in $G'$. Either way, $y$ cannot force any of its neighbors in $G$. We also see that $w$ cannot force its neighbors since $w$ is adjacent to $v$ and $x$ which are both unfilled. Thus, $S$ is stalled. Note $|S' \cup \{w\}| \ge \lfloor \frac{(n - 2) - 1}{2} \rfloor + 1 = \lfloor \frac{n - 1}{2} \rfloor$.
\begin{figure}[H]
    \caption{The graph $G$ with $\delta(G) = 1$. Case where $w$ is adjacent to only vertices in $S' \cup v$.}
    \centering
    \begin{tikzpicture}
        \node[unfilled, label=below:{$v$}] at (0,0) (v) {};
        \node[unfilled, label=below:{$w$}] at (1,0) (w) {};
        \node[filled] at (2, -1) (w_1) {};
        \node[filled] at (2, 0) (w_2) {};
        \node[filled] at (2, 1) (w_3) {};

        \draw (v) -- (w);
        \draw (w) -- (w_1);
        \draw (w) -- (w_2);
        \draw (w) -- (w_3);

        \node[collection, minimum height = 4cm, minimum width = 3.5cm, dotted] at (3.2, 0) (W) {$G'$};
    \end{tikzpicture}
    \begin{tikzpicture}
        \node[filled, label=below:{$v$}] at (0,0) (v) {};
        \node[filled, label=below:{$w$}] at (1,0) (w) {};
        \node[filled] at (2, -1) (w_1) {};
        \node[filled] at (2, 0) (w_2) {};
        \node[filled] at (2, 1) (w_3) {};

        \draw (v) -- (w);
        \draw (w) -- (w_1);
        \draw (w) -- (w_2);
        \draw (w) -- (w_3);

        \node[collection, minimum height = 4cm, minimum width = 3.5cm, dotted] at (3.2, 0) (W) {$G'$};
    \end{tikzpicture}
\end{figure}
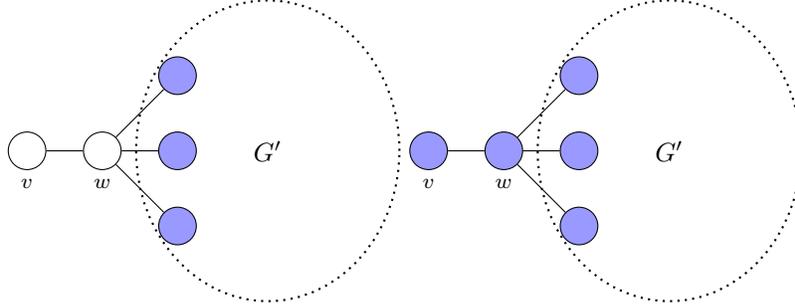
If $w$ is adjacent to only filled vertices in $G'$, we see then $S = S' \cup \{v, w\}$ is stalled. This is because $v$ and $w$ do not cause any forcing in $G'$. Since $S' \ne V'$, we have $S \ne V$. So $S$ is a failed forcing set with $|S| = |S'| + 2 \ge \lfloor \frac{n + 1}{2} \rfloor$.

In either case, we find a stalled set of size at least $\lfloor \frac{n - 1}{2} \rfloor$.

Now let $\delta(G) = 2$. Then there exists some vertex $v \in V$ such that $\deg(v) = 2$. Let $x,y$ be the neighbors of $v$. Let $X$ be the set of all the vertices in $V - \{v\}$  adjacent to $x$. Let $Y$ be the set of all the vertices in $V - \{v\}$ adjacent to $y$.
We construct a graph $G'$ by condensing  the $x-v-y$ subgraph to a single vertex $w$. More precisely, let $G'=(V',E')$ be obtained from the subgraph of $G$ induced by $V-\{v,x,y\}$, by adding a new vertex $w$, and adding the edge $(w,z)$ for each $z\in X\cup Y$.
\begin{figure}[H]
    \caption{Graph $G$ (left) with $\delta(G) = 2$ with corresponding $G'$ (right).}
    \centering
    \begin{tikzpicture}
        \node[unfilled, label=above:{$v$}] at (0,0) (v) {};
        \node[unfilled, label=above:{$x$}] at (-1,-1) (x) {};
        \node[unfilled, label=above:{$y$}] at (1,-1) (y) {};
        \node[] at (-1.5, -2) (x_1) {};
        \node[] at (-0.5, -2) (x_2) {};
        \node[] at (1.5, -2) (y_1) {};
        \node[] at (0.5, -2) (y_2) {};

        \draw (y) -- (v) -- (x);
        \draw (x) -- (x_1);
        \draw (x) -- (x_2);
        \draw (y) -- (y_1);
        \draw (y) -- (y_2);

        \node[collection,  minimum width = 2.5cm, minimum height = 0.75cm] at (-1, -2) (X) {$X$};
        \node[collection,  minimum width = 2.5cm, minimum height = 0.75cm] at (1, -2) (Y) {$Y$};
    \end{tikzpicture}
    \begin{tikzpicture}
        \node[unfilled, label=above:{$w$}] at (0,0) (w) {};

        \node[] at (-1.5, -1) (x_1) {};
        \node[] at (-0.5, -1) (x_2) {};
        \node[] at (1.5, -1) (y_1) {};
        \node[] at (0.5, -1) (y_2) {};

        \draw (w) -- (x_1);
        \draw (w) -- (x_2);
        \draw (w) -- (y_1);
        \draw (w) -- (y_2);
        \node[collection, minimum width = 4cm, minimum height = 0.75cm] at (0, -1) (XY) {$X \cup Y$};
    \end{tikzpicture}
\end{figure}
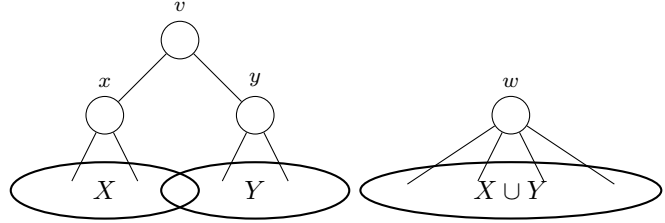

 We inductively assume $F(G') \ge \lfloor \frac{(n - 2) - 1}{2} \rfloor$. Let $S'$ be a corresponding stalled set. We consider the following cases:

\begin{enumerate}
\item
\begin{figure}[H]
    \caption{Case when $w \not \in S'$ and corresponding stalled set in $G$.}
    \centering
    \begin{tikzpicture}
        \node[unfilled, label=above:{$w$}] at (0,0) (w) {};

        \node[] at (-1.5, -1) (x_1) {};
        \node[] at (-0.5, -1) (x_2) {};
        \node[] at (1.5, -1) (y_1) {};
        \node[] at (0.5, -1) (y_2) {};

        \draw (w) -- (x_1);
        \draw (w) -- (x_2);
        \draw (w) -- (y_1);
        \draw (w) -- (y_2);
        \node[collection, minimum width = 4cm, minimum height = 0.75cm] at (0, -1) (XY) {};
    \end{tikzpicture}
    \begin{tikzpicture}
        \node[filled, label=above:{$v$}] at (0,0) (v) {};
        \node[unfilled, label=above:{$x$}] at (-1,-1) (x) {};
        \node[unfilled, label=above:{$y$}] at (1,-1) (y) {};
        \node[] at (-1.5, -2) (x_1) {};
        \node[] at (-0.5, -2) (x_2) {};
        \node[] at (1.5, -2) (y_1) {};
        \node[] at (0.5, -2) (y_2) {};

        \draw (y) -- (v) -- (x);
        \draw (x) -- (x_1);
        \draw (x) -- (x_2);
        \draw (y) -- (y_1);
        \draw (y) -- (y_2);

        \node[collection,  minimum width = 2.5cm, minimum height = 0.75cm] at (-1, -2) (X) {};
        \node[collection,  minimum width = 2.5cm, minimum height = 0.75cm] at (1, -2) (Y) {};
    \end{tikzpicture}
\end{figure}
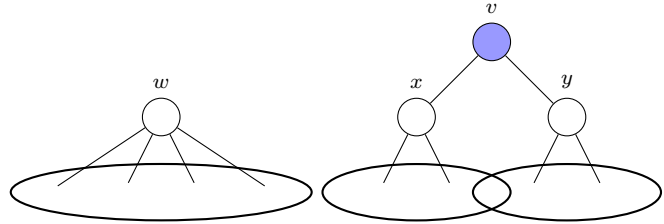
If $w \not \in S'$, then $S = S' \cup \{v\}$ is stalled in $G$. We see that any $z \in S'$ that could not force $G'$ also cannot force in $G$ since $x$ and $y$ are unfilled to match the unfilled $w$. We also see that $v$ cannot zero force in $G$ since $v$ is adjacent to $x,y \not \in S$. Thus, $S$ is a stalled set.
    \item If $w \in S'$, it is either adjacent to only filled vertices or adjacent to at least two unfilled vertices in $G'$.
\begin{enumerate}
    \item
    \begin{figure}[H]
    \caption{Case when $w$ is adjacent to filled vertices in $G'$ and corresponding stalled set in $G$.}
    \centering
    \begin{tikzpicture}
        \node[filled, label=above:{$w$}] at (0,0) (w) {};

        \node[filled] at (-1.5, -1) (x_1) {};
        \node[filled] at (-0.5, -1) (x_2) {};
        \node[filled] at (1.5, -1) (y_1) {};
        \node[filled] at (0.5, -1) (y_2) {};

        \draw (w) -- (x_1);
        \draw (w) -- (x_2);
        \draw (w) -- (y_1);
        \draw (w) -- (y_2);
        \node[collection, minimum width = 4cm, minimum height = 0.75cm] at (0, -1) (XY) {};
    \end{tikzpicture}
    \begin{tikzpicture}
        \node[filled, label=above:{$v$}] at (0,0) (v) {};
        \node[filled, label=above:{$x$}] at (-1,-1) (x) {};
        \node[filled, label=above:{$y$}] at (1,-1) (y) {};
        \node[filled] at (-1.5, -2) (x_1) {};
        \node[filled] at (-0.5, -2) (x_2) {};
        \node[filled] at (1.5, -2) (y_1) {};
        \node[filled] at (0.5, -2) (y_2) {};

        \draw (y) -- (v) -- (x);
        \draw (x) -- (x_1);
        \draw (x) -- (x_2);
        \draw (y) -- (y_1);
        \draw (y) -- (y_2);

        \node[collection,  minimum width = 2.5cm, minimum height = 0.75cm] at (-1, -2) (X) {};
        \node[collection,  minimum width = 2.5cm, minimum height = 0.75cm] at (1, -2) (Y) {};
    \end{tikzpicture}
    \end{figure}
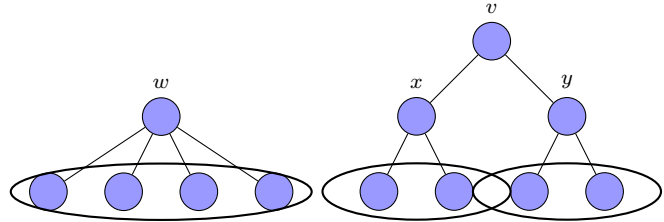
    If $w$ is adjacent to only  filled vertices in $G'$, then $S = (S' - \{w\}) \cup \{v,x,y\}$ is a stalled set in $G$. If a vertex was spent in $S'$, we see it must also be spent in $S$ since both $x$ and $y$ get filled. If a vertex $z$ was adjacent to at least two unfilled vertices in $S'$, it must also be adjacent to at least two unfilled vertices in $S$ as $z$ could not not have been adjacent to $x$ or $y$. Thus, $S$ is stalled.
\item
If $w$ is adjacent to at least two vertices not in $S'$, we will fill $x$ and $y$ in $G$. We first observe that every vertex in $(S' - \{w\})$ in $G$ is spent or adjacent to at least two unfilled vertices. If a vertex $z \in S'$ was adjacent to $w$, then $z$ is adjacent to $x$ or $y$ or both in $G$. Since we fill $x$ and $y$ to match $w$ being filled, we see $z$ cannot force in $G'$. Since $w$ is adjacent to at least two unfilled vertices in $G'$, at least two vertices in $X \cup Y$ are not in $S'$.
\begin{figure}[H]
    \caption{Case when $w$ is adjacent to exactly one unfilled vertex in $X$ and one in $Y$. Corresponding stalled set in $G$.}
    \label{fig:oneXoneY}
    \centering
    \begin{tikzpicture}
        \node[filled, label=above:{$w$}] at (0,0) (w) {};

        \node[filled] at (-1.5, -1) (x_1) {};
        \node[unfilled] at (-0.5, -1) (x_2) {};
        \node[filled] at (1.5, -1) (y_1) {};
        \node[unfilled] at (0.5, -1) (y_2) {};

        \draw (w) -- (x_1);
        \draw (w) -- (x_2);
        \draw (w) -- (y_1);
        \draw (w) -- (y_2);
        \node[collection, minimum width = 4cm, minimum height = 0.75cm] at (0, -1) (XY) {};
    \end{tikzpicture}
    \begin{tikzpicture}
        \node[unfilled, label=above:{$v$}] at (0,0) (v) {};
        \node[filled, label=above:{$x$}] at (-1,-1) (x) {};
        \node[filled, label=above:{$y$}] at (1,-1) (y) {};
        \node[filled] at (-1.5, -2) (x_1) {};
        \node[unfilled] at (-0.5, -2) (x_2) {};
        \node[filled] at (1.5, -2) (y_1) {};
        \node[unfilled] at (0.5, -2) (y_2) {};

        \draw (y) -- (v) -- (x);
        \draw (x) -- (x_1);
        \draw (x) -- (x_2);
        \draw (y) -- (y_1);
        \draw (y) -- (y_2);

        \node[collection,  minimum width = 2.5cm, minimum height = 0.75cm] at (-1, -2) (X) {};
        \node[collection,  minimum width = 2.5cm, minimum height = 0.75cm] at (1, -2) (Y) {};
    \end{tikzpicture}
\end{figure}
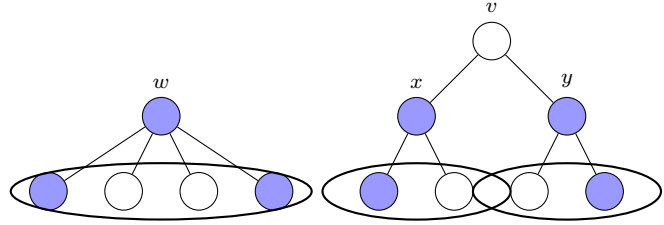
\begin{enumerate}
    \item 
    If there is at least one vertex in $X - S'$ and at least one vertex in $Y - S'$, then both $x$ and $y$ are adjacent to at least two unfilled vertices in $G$ since both $x$ and $y$ are adjacent to some unfilled vertex in $(X \cup Y)$ as well as $v$. This means $S = (S' - \{w\}) \cup \{x, y\}$ is a stalled set in $G$. (As seen in Figure \ref{fig:oneXoneY}).
    \item
    \begin{figure}[H]
    \caption{Case when $w$ is adjacent to at least two unfilled vertices in $X$ or $Y$. Corresponding stalled set in $G$.}
    \centering
    \begin{tikzpicture}
        \node[filled, label=above:{$w$}] at (0,0) (w) {};

        \node[filled] at (-1.5, -1) (x_1) {};
        \node[filled] at (-0.5, -1) (x_2) {};
        \node[unfilled] at (1.5, -1) (y_1) {};
        \node[unfilled] at (0.5, -1) (y_2) {};

        \draw (w) -- (x_1);
        \draw (w) -- (x_2);
        \draw (w) -- (y_1);
        \draw (w) -- (y_2);
        \node[collection, minimum width = 4cm, minimum height = 0.75cm] at (0, -1) (XY) {};
    \end{tikzpicture}
    \begin{tikzpicture}
        \node[filled, label=above:{$v$}] at (0,0) (v) {};
        \node[filled, label=above:{$x$}] at (-1,-1) (x) {};
        \node[filled, label=above:{$y$}] at (1,-1) (y) {};
        \node[filled] at (-1.5, -2) (x_1) {};
        \node[filled] at (-0.5, -2) (x_2) {};
        \node[unfilled] at (1.5, -2) (y_1) {};
        \node[unfilled] at (0.5, -2) (y_2) {};

        \draw (y) -- (v) -- (x);
        \draw (x) -- (x_1);
        \draw (x) -- (x_2);
        \draw (y) -- (y_1);
        \draw (y) -- (y_2);

        \node[collection,  minimum width = 2.5cm, minimum height = 0.75cm] at (-1, -2) (X) {};
        \node[collection,  minimum width = 2.5cm, minimum height = 0.75cm] at (1, -2) (Y) {};
    \end{tikzpicture}
\end{figure}
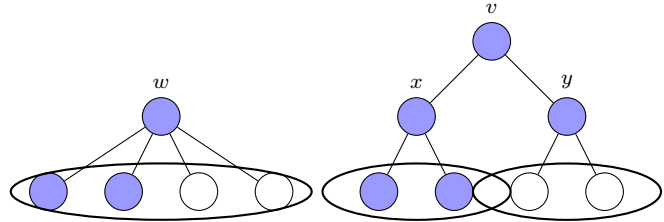
    If $X - S' = \emptyset$ or $Y - S' = \emptyset$, we claim $S = (S' - \{w\}) \cup \{v, x, y\}$ is stalled in $G$. Without loss of generality, assume $X - S' = \emptyset$ and $|Y - S'| \ge 2$. Note $x$ has neighbors $X \cup \{v\}$ in $G$. Since $v \in S$ and $X \subset S$, all of $x$'s neighbors are in $S$. Also, since $x, y \in S$, all of $v$'s neighbors are in $S$. Finally, $y$ has at least two vertices in $V - S$ since $|Y - S'| \ge 2$. Thus, none of $x, y$, or $v$ can force, making $S$ a stalled set.
\end{enumerate}
\end{enumerate}
\end{enumerate}
In any case, we see we can fill at least one more vertex in $G$ than we could $G'$. That is, $|S| \ge |S'| + 1$ and thus $|S| \ge \lfloor \frac{n-1}{2} \rfloor$.
\end{proof}

\section{Conclusions and Future Work}

We now know there are finitely many graphs where $F(G) = k$ for any given $k$ and we can enumerate them by checking all graphs with fewer than $2k + 2$ vertices. For $k\leq 4$ we can do this fairly quickly by computing $F(G)$  for all graphs $G$ with no more than ten vertices.  For $k > 4$ the number of graphs one must check gets very large. 
We let $F_k$ be the set of all connected graphs $G$ such that $F(G) = k$, and $F_k(n)$ be the number of graphs in $F_k$ on $n$ vertices. Using a exhaustive search on all connected graphs up to 10 vertices we get:
\begin{figure}[H]
    \caption{}
    \centering
    \begin{tabular}{c|cccccccc}
     $n =$ & 3& 4& 5& 6& 7& 8& 9& 10  \\
     \hline
     $F_1(n)$ & 2& 1&  &  &  &  &  & \\
     $F_2(n)$ &  & 5& 5& 2&  &  &  & \\
     $F_3(n)$ &  &  &16&29&16& 1&  & \\
     $F_4(n)$ &  &  &  &81&277&268&14& 1\\
\end{tabular}
\end{figure}
The rows of the table correspond to graphs in $F_k$ while the columns correspond to graphs with $n$ vertices. For example, $F_1$ contains three graphs: two with three vertices and one with four vertices. We can sum the rows to get $|F_2| = 12$, $|F_3| = 62$, and $|F_4| = 641$. This confirms the result by \cite{gomez} where it was shown that there are $15$ graphs with a failed zero forcing number of two, twelve of which were connected.


We can apply our lower bound to another question problem posed in $\cite{fetcie}$: the classification of graphs that satisfy $Z(G) = F(G)$. For small $k$ we can enumerate all graphs where $F(G) = k$, and see which of these satisfy $Z(G) = k$. This gives us an exhaustive list of $G$ for which $F(G) = Z(G) = k$. We define $E_k$ to be the set of all connected graphs $G$ such that $F(G) = Z(G) = k$, and   $E_k(n)$ to be the number of graphs on $n$ vertices in $E_k$. We have computed $E_k(n)$ for $1 \le k \le 4$ and have found the following:\\
\begin{figure}[H]\label{ekn}
    \caption{}
    \centering
    \begin{tabular}{c|cccccccc}
     $n = $ & 3 & 4 & 5 & 6 & 7 & 8 & 9 & 10  \\
     \hline
     $E_1(n)$ & 1& 1&  &  &  &  &  & \\
     $E_2(n)$ &  & 4& 4& 1&  &  &  & \\
     $E_3(n)$ &  &  & 9&10& 4&  &  & \\
     $E_4(n)$ &  &  &  &19&29& 2&  & \\
\end{tabular}
\end{figure}
 The columns of the table correspond to graphs with $n$ vertices; the sum across a row is then $|E_k|$.  For example, the pair of $1$s in the first row corresponds to two graphs in $E_1$, one with three vertices and the other with four. These graphs are $P_3$ and $P_4$, the only graphs such that $F(G) = Z(G) = 1$. Furthermore, we see that $|E_1| = 2$, $|E_2| = 9$, $|E_3| = 23$, and $|E_4| = 50$. In \cite{fetcie} it was claimed that graphs where $F(G) = Z(G)$ seem to be rare. The above gives us a quantitative method to evaluate this claim,  by finding the number of graphs with a fixed failed zero forcing number $k$ that have the same zero forcing number. For example, we see that among the 641 graphs with failed zero forcing number of four, exactly 50 have zero forcing number four.

Interesting questions to consider include finding bounds on $|E_k|$ and characterizing $n$ for which $E_k(n)\neq 0$. Toward the latter question, we note $k = F(G) \le n - 2$ for any connected graph $G$ and so any graph with $n$ vertices in $E_k$ must satisfy $n \ge k + 2$. This inequality gives us the lower diagonal in Figure 12 and Figure 13. However, we see an upper diagonal in Figure 13 that is stricter than the one in Figure 12. That is, given a connected graph $G$ with $n$ vertices in $E_k$, we ask whether $n \le k + 4$. We have checked this holds true for all connected graphs $G$ up to ten vertices.
%
%
%

\begin{thebibliography}{}


\bibitem{barioli} Barioli, F.; Barrett, W.; Butler, S.; Cioaba, S. M.; Cvetkovi, D.; Fallat,  S. M.; Godsil, C.; Haemers, W.;Hogben, L.; Mikkelson, R.; Narayan, S.;   Pryporova, O.; Sciriha,  I.; So, W.; Stevanovic, D.; van der Holst, H.; Vander Meulen, K.;  and Wangsness,  A.  Zero forcing sets and the minimum rank of graphs, Linear Algebra Appl. 428:7 (2008), 1628–1648.  

\bibitem{fetcie} {Fetcie, K.; Jacob, B.; Saavedra, D. The failed zero forcing number
of a graph. Involve 2015,8, 1, 99--118}

\bibitem{gomez} Gomez,L.; Rubi,K.; Terrazas, J.; Narayan, D.A. All Graphs with a Failed Zero Forcing Number of Two. Symmetry 2021, 13, 2221. https://doi.org/10.3390/sym13112221
 
\end{thebibliography}


%
\end{document}